\documentclass[11pt]{article}

\usepackage{amsfonts, amsmath, amssymb, amscd, amsthm, color, graphicx, mathrsfs, wasysym, setspace, mdwlist, calc, float, xcolor}
\usepackage{setspace}
\usepackage{}

\usepackage{hyperref}
\hypersetup{linktocpage}
\usepackage[shortlabels]{enumitem}

\hypersetup{colorlinks,
    linkcolor={red!50!black},
    citecolor={blue!80!black},
    urlcolor={blue!80!black}}
\usepackage{float}

\newcommand{\NN}{\mathbb{N}}
\newcommand{\ZZ}{\mathbb{Z}}
\newcommand{\RR}{\mathbb{R}}
\newcommand{\QQ}{\mathbb{Q}}

\renewcommand{\ll }{\langle\hspace{-.7mm}\langle }
\newcommand{\rr }{\rangle\hspace{-.7mm}\rangle }
\newcommand{\bll }{\Big\langle\hspace{-.7mm}\Big\langle }
\newcommand{\brr }{\Big\rangle\hspace{-.7mm}\Big\rangle }

\newtheorem{thm}{Theorem}[section]
\newtheorem{cor}[thm]{Corollary}
\newtheorem{lem}[thm]{Lemma}
\newtheorem{prop}[thm]{Proposition}
\newtheorem{prob}[thm]{Problem}

\theoremstyle{definition}
\newtheorem{defn}[thm]{Definition}

\theoremstyle{remark}
\newtheorem{rem}[thm]{Remark}
\newtheorem{ex}[thm]{Example}

\renewcommand{\d}{{\rm d}}

\begin{document}

\title{\vspace*{-10mm} Frattini subgroups of hyperbolic-like groups}
\date{}

\author{G. Goffer, D. Osin, E. Rybak}

\maketitle

\begin{abstract}
We study Frattini subgroups of various generalizations of hyperbolic groups. For any countable group $G$ admitting a general type action on a hyperbolic space $S$, we show that the induced action of the Frattini subgroup $\Phi(G)$ on $S$ has bounded orbits. This implies that $\Phi(G)$ is ``small" compared to $G$; in particular, $|G:\Phi(G)|=\infty$. In contrast, for any finitely generated non-cyclic group $Q$ with $\Phi(Q)=\{ 1\}$, we construct an infinite lacunary hyperbolic group $L$ such that $L/\Phi(L)\cong Q$; in particular, the Frattini subgroup of an infinite lacunary hyperbolic group can have finite index. As an application, we obtain the first examples of invariably generated, infinite, lacunary hyperbolic groups.
\end{abstract}

\section{Main results}

The \textit{Frattini subgroup} of a group $G$, denoted by $\Phi(G)$, is the intersection of all maximal subgroups of $G$; if $G$ has no maximal subgroups, $\Phi(G)=G$ by definition. Equivalently, $\Phi(G)$ can be defined as the set of all non-generating elements of $G$. Recall that an element $g\in G$ is \textit{non-generating} if it satisfies the following condition: for any generating set $X$ of $G$, $X\setminus\{ g\}$ also generates $G$.  

The study of Frattini subgroups has a long history. It is well-known that $\Phi(G)$ is nilpotent whenever it is finite. In particular, the Frattini subgroup of a finite group is nilpotent. Platonov \cite{P} and Wehrfritz \cite{W} independently extended this result to finitely generated linear groups. Finite groups that can be realized as $\Phi(G)$ for a certain group $G$ are characterized in terms of their automorphism groups by Eick in \cite{Eic}. Note, however, that the Frattini subgroup of a finitely generated group need not be finite in general. For instance, results of Pervova \cite{Per} imply that the Grigorchuk $2$-group has finite index Frattini subgroup. Furthermore, answering a question of Wiegold, Obraztsov \cite{Obr} showed that $\Phi(G)$ can be infinite and simple for a suitable finitely generated group $G$.

Our paper is motivated by the observation that the Frattini subgroup of groups with ``hyperbolic-like" geometry is often small in a suitable sense. For instance, it is not difficult to show that $\Phi(G)=\{1\}$ whenever $G$ is the free product of two non-trivial groups, as shown by Higman and Neumann \cite{HN}. Long \cite{L} proved that the Frattini subgroup of the mapping class group of a closed hyperbolic surface is isomorphic to $\mathbb Z_2$ if the genus of the surface equals $2$ and is trivial otherwise. Kapovich \cite{K} showed that $\Phi(G)$ is finite for any non-elementary subgroup $G$ of a hyperbolic group; this was generalized to all countable non-elementary convergence groups by Gelander and Glasner \cite{GG}. A generalization of all these results to the class of acylindrically hyperbolic groups was obtained by Hull \cite{H}. The ``smallness" of the Frattini subgroup was also proved for countable groups acting on trees (see \cite{All,GG} and references therein).

To state our first theorem, we need to recall a definition from \cite{Gro}; for details and unexplained terminology, we refer to the next section. 

\begin{defn}\label{def}
Let $S$ be a hyperbolic space. A subgroup $G\le \operatorname{Isom}(S)$ is said to be \textit{non-elementary} if the limit set of $G$ contains more than $2$ points. If, in addition, $G$ does not fix any point of the Gromov boundary $\partial S$, we say that $G$ has \textit{general type}. 
\end{defn}

For example, the action of a hyperbolic group $G$ on its Cayley graph with respect to a finite generating set is non-elementary if and only if $G$ is not virtually cyclic. It is also easy to show that the action of an amalgamated free product $A\ast _CB$ on the associated Bass-Serre tree is non-elementary if and only if $A\ne C \ne B$ and $C$ has index at least $3$ in at least one of the groups $A$, $B$.

We begin by proving the following.

\begin{thm}\label{main0} Let $S$ be a hyperbolic space. For any countable, general type subgroup $G$ of $\operatorname{Isom}(S)$, the action of $\Phi(G)$ on $S$ has bounded orbits. In particular, $|G:\Phi(G)|=\infty$.
\end{thm}

It is not difficult to show that Theorem \ref{main0} implies the results of \cite{GG,H,K,L} mentioned above (see Corollary \ref{corPC}) as well as the main result of \cite{All} in the particular case of countable groups. In fact, our theorem can be derived from \cite[Theorem 3.3]{GG} and some general facts about groups acting on hyperbolic spaces. However, we choose to provide an independent proof, which is shorter and more elementary than that of \cite[Theorem 3.3]{GG}. It is also worth noting that the conclusion of Theorem \ref{main0} may not hold for non-elementary actions as demonstrated in Example \ref{ex}.

Further, we consider lacunary hyperbolic groups introduced in \cite{OOS}. Recall that a group $G$ is \emph{lacunary hyperbolic} if it is finitely generated and at least one asymptotic cone of $G$ is an $\RR$-tree. In conjunction with finite presentability, the latter condition is known to imply hyperbolicity of $G$. Thus, the class of lacunary hyperbolic groups can be thought of as a generalization of the class of hyperbolic groups. It is shown in \cite{OOS} that these two classes have similar algebraic properties. However, this similarity does not extend to the Frattini subgroup. 

\begin{thm}\label{main1}
Suppose that a finitely generated group $Q$ is not finite cyclic and has a trivial Frattini subgroup. Then there exists an infinite lacunary hyperbolic group $L$ such that $L/\Phi(L)\cong Q$. 
\end{thm}

For a group $G$, we call $G/\Phi(G)$ the \textit{Frattini quotient} of $G$. Using the definition of $\Phi(G)$ as the set of non-generating elements, it is easy to show that if $G$ is finitely generated, then $G/\Phi(G)$ cannot be finite cyclic unless $G$ is itself finite cyclic. It is also clear that $\Phi(G/\Phi(G))=\{ 1\}$. Therefore, the assumptions about $Q$ in Theorem \ref{main1} cannot be relaxed.

\begin{rem}
The proof of Theorem \ref{main1} employs results obtained via small cancellation theory in hyperbolic groups outlined by Gromov \cite{Gro} and elaborated by Olshanskii \cite{Ols}. By using small cancellation theory in relatively hyperbolic groups instead (see \cite{Osi10}), we could additionally ensure that $\Phi(L)$ contains any given lacunary hyperbolic group as a subgroup. Furthermore, utilizing small cancellation theory in acylindrically hyperbolic groups developed in \cite{H}, one can prove the following: \textit{For any finitely generated, non-cyclic group $Q$ with trivial Frattini subgroup and any countable group $H$, there exists an infinite, finitely generated group $G$ such that $G/\Phi(G)\cong Q$ and $H\le \Phi(G)$.} The proofs of these results are more technical and we restrict our attention to Theorem~\ref{main1} in this paper.
\end{rem}

Taking $Q$ to be a non-cyclic abelian group of order $n$ we obtain the following corollary, which sharply contrasts Theorem \ref{main0}. 

\begin{cor}\label{cor}
For every integer $n$ divisible by a square of a prime, there exists an infinite lacunary hyperbolic group $L$ such that $|L:\Phi(L)|=n$ and all maximal subgroups of $L$ are normal of finite index.
\end{cor}

Note that we need $n$ to be divisible by a square of a prime to ensure that there exists a non-cyclic abelian group of order $n$. More precisely, if $n=p_1^{\alpha_1}p_2^{\alpha_2} \cdots p_k^{\alpha_k}$, where $p_1, p_2, \ldots, p_k$ are primes, $\alpha _1, \alpha_2, \ldots, \alpha_k\in \NN$, and $\alpha _i\ge 2$ for at least one $i$, then the group $Q=\bigoplus_{i=1}^k (\ZZ/p_i\ZZ)^{\alpha_i}$ is non-cyclic and has trivial Frattini subgroup, so Theorem \ref{main1} applies. 

As another application, we obtain the first examples of infinite, invariably generated, lacunary hyperbolic groups. Recall that a subset $S$ of a group $G$ is an \emph{invariable generating set} of $G$ if for every function $f\colon S\to G$, the group $G$ is generated by the subset $\{f(s)^{-1}sf(s) \mid s\in S\}$. A group $G$ is \textit{invariably generated} if it admits an invariable generating set (equivalently if $S=G$ is an invariable generating set of $G$); if $G$ has a finite invariable generating set, it is said to be \textit{finitely invariably generated}. This terminology was coined by Dixon \cite{D} in 1988, although equivalent properties of groups have been studied since the 19th century.

Every finite group is invariably generated by a result of Jordan \cite{J}. For infinite groups, the phenomenon of invariable generation typically occurs in non-hyperbolic settings. For example, results of Wiegold \cite{Wie} imply that every virtually solvable group is invariable generated; furthermore, a linear group is finitely invariably generated if and only if it is finitely generated and virtually solvable, as shown by Kantor, Lubotzky, and Shalev \cite{KLS}. On the other hand, non-elementary hyperbolic groups and, more generally, non-elementary convergence groups are not invariably generated \cite{G}.  

The property of invariable generation is closely related to Frattini quotients. For instance, if $G$ is finitely generated and $G/\Phi(G)$ is finite, then $G$ is finitely invariably generated (see \cite[Lemma 2.6]{KLS}). In the recent paper \cite{CT}, Cox and Thillaisundaram showed that every generating set of a finitely generated group $G$ is an invariable generating set if and only if all maximal subgroups of $G$ are normal; in particular, both properties hold whenever $G/\Phi(G)$ is abelian. Combining the latter result with Theorem \ref{main1}, we obtain the following.

\begin{cor}
    There exists an infinite lacunary hyperbolic group $L$ such that every generating set of $L$ is an invariable generating set. In particular, $L$ is finitely invariably generated.
\end{cor}

\paragraph{Acknowledgments.}  D. Osin was supported by the NSF grant DMS-1853989. We would also like to thank the anonymous referee for the careful reading of our manuscript and useful remarks.


\section{Groups acting on hyperbolic spaces}


We begin by briefly reviewing the standard background on groups acting on hyperbolic spaces. Our main references are \cite{book,Gro,Ham}. We stress that we do not assume our hyperbolic spaces to be proper. All group actions on metric spaces considered in this paper are assumed to be isometric by default.

Throughout this section, let $(S, d)$ be a hyperbolic metric space. Recall that the \textit{Gromov product} of two points $x, y \in S$ with respect to a point $z \in S$ is defined by
$$
(x, y)_z=\frac{1}{2}(\mathrm{~d}(x, z)+\mathrm{d}(y, z)-\mathrm{d}(x, y)).
$$
A sequence $\left(x_i\right)$ of elements of $S$ \textit{converges at infinity} if $\left(x_i, x_j\right)_s \rightarrow \infty$ as $i, j \rightarrow \infty$ for some, or equivalently for any point $s\in S$. Two such sequences $\left(x_i\right)$ and $\left(y_i\right)$ are \textit{equivalent} if $\left(x_i, y_j\right)_s \rightarrow \infty$ as $i, j \rightarrow \infty$. The Gromov boundary of $S$, denoted by $\partial S$, is defined as the set of equivalence classes of sequences converging at infinity. If $x$ is the equivalence class of $\left(x_i\right)$, we say that the sequence $(x_i)$ converges to $x$. This naturally defines a topology on $\widehat{S}=S \cup \partial S$ extending the topology on $S$ such that $S$ is dense in $\widehat{S}$. It is not difficult to show that the definition of $\partial S$ is independent of the choice of the basepoint and any isometric group action on $S$ extends to a continuous action on $\widehat S$ (see \cite[Chapter 3]{book}).

For a group $G$ acting on $S$ and any $s\in S$, let $Gs=\{ gs\mid g\in G\}$ denote the $G$-orbit of the point $s$. The \textit{limit set} $\Lambda(G)$ of $G$ is defined to be the intersection of the closure of $Gs$ in $\widehat S$ with $\partial S$. It is easy to see that this definition is independent of the choice of $s\in S$. 

An element $g \in G$ is said to be {\it loxodromic} if the group $\langle g\rangle $ has unbounded orbits and fixes exactly $2$ points on $\partial S$ (equivalently, $|\Lambda(\langle g\rangle)|=2$). The following lemma establishes the standard notation for fixed points of loxodromic elements motivated by the so-called north-south dynamics (see \cite[Lemma 8.1.G]{Gro} or \cite[Theorem 6.1.10]{book}).

\begin{lem} \label{north-south}
The fixed points of a loxodromic element $g \in G$ can be denoted by $g^+$ and $g^-$ so that
$$
\lim _{n \rightarrow \infty} g^n s=g^{+} \;\;\;  \forall s \in \widehat S \setminus \left\{g^{-}\right\}\;\;\;\;\; {\rm and }\;\;\;\;\; \lim _{n \rightarrow \infty} g^{-n} s=g^{-} \;\;\;  \forall s \in \widehat S \setminus \left\{g^{+}\right\}
$$
where the convergence in the first (respectively, second) limit is uniform on sets whose closure does not contain $g^{-}$ (respectively, $g^+$).
\end{lem}

\begin{rem} \label{congugation}
It is easy to see that if $g$ is loxodromic, then so is $fgf^{-1}$ for any $f\in G$. Moreover, we have $(fgf^{-1})^{+}=fg^+$, and  $(fgf^{-1})^{-}=fg^-$.
\end{rem}

Recall that loxodromic elements $g_1,g_2 \in G$ are {\it independent} if $\left\{g_1^{+}, g_1^{-}\right\} \cap \left\{g_2^{+}, g_2^{-}\right\} =\emptyset$. If $G\le Isom(S)$ is of general type, we can find infinitely many independent loxodromic elements in $G$.  Conversely, the existence of two independent loxodromic elements in $G$ implies that the action of $G$ on $S$ is of general type (see \cite[Section 2]{Ham} for details).

The following proposition is the main technical result of this section. It seems to be of independent interest. 

\begin{prop} \label{main_prop}
Let $S$ be a hyperbolic space. For any countable set of isometries $x_1, x_2, ...$ of $S$ and any general type subgroup $H\le \operatorname{Isom}(S)$, there exist $a_1, b_1, a_2, b_2, ... \in H$ such that the elements $b_1x_1a_1, b_2x_2a_2, ...$ form a basis of a free subgroup $F \le \operatorname{Isom}(S)$ that does not contain $H$. 
\end{prop}

\begin{rem}
    It is worth noting that, for the proof of Theorem \ref{main0}, we only need the conclusion $H\not\subseteq F$; however, the assertion that $F$ is free is obtained as a by-product of the proof without any additional arguments. 
\end{rem}

\begin{proof}[Proof of Proposition \ref{main_prop}]
Since $H$ is a general type subgroup of $\operatorname{Isom}(S)$, it contains two independent loxodromic elements  $f$ and $g$. We will first construct auxiliary elements $a_1, b_1, a_2, b_2, \ldots \in \langle f, g \rangle$ and a sequence of nonempty pairwise disjoint subsets $A_1^+, A_1^-, B_1^+, B_1^-, A_2^+, A_2^-, B_2^+, B_2^-, \ldots $ of $\widehat S$ such that the following conditions hold for all $i\in \NN$.

\begin{enumerate}
    \item[(a)]  $a_i(\widehat{S}\setminus A_i^-)\subseteq A_i^+$, $a_i^{-1}(\widehat{S}\setminus A_i^+)\subseteq A_i^-$, $b_i(\widehat{S}\setminus B_i^-)\subseteq B_i^+$, and  $b_i^{-1}(\widehat{S}\setminus B_i^+)\subseteq B_i^-$.
    \item[(b)] $x_i(A_i^+)\cap B_i^{-} =\emptyset$ (equivalently, $A_i^+\cap x^{-1}_i (B_i^-)=\emptyset$).
    \item[(c)] $ f^+ \not\in \overline{A_i^+} \cup \overline{A_i^{-}} \cup \overline{B_i^+}\cup \overline{B_i^-}$.
\end{enumerate}

We proceed by induction on $i$. The base is similar to the inductive step, so we treat them simultaneously. Let $n\geq 1$. Suppose we have already constructed elements $a_i, b_i$ and sets $A_i^+, A_i^-, B_i^+, B_i^-$ satisfying conditions (a)--(c) for all $i < n$. Part (c) of the inductive assumption implies that the open set
$$
U_n=\widehat S\setminus\bigcup_{i=1}^{n-1} \left(\overline{A_i^+}\cup \overline{A_i^-}\cup \overline{B_i^+}\cup \overline{B_i^-}\right)
$$
contains $f^+$. Consider an element $k_n = f^{s_n}gf^{-s_n}$ for some $s_n \in \NN$. By Remark \ref{congugation},  $k_n$ is loxodromic with $k_n^+ =f^{s_n}g^+$ and $ k_n^- =f^{s_n}g^-$. Since $f$ and $g$ are independent, we have $k_n^\pm \ne f^+$ and $k_n^\pm \to f^+$ as $s_n\to \infty$ by Lemma \ref{north-south}. Therefore, there exists $s_n\in \NN $ such that 
\begin{equation}\label{Eq:kn}
\{ k_n^+,  k_n^-\} \in U_n\setminus\{ f^+\}.
\end{equation}
Similarly, we can find $t_n \in \NN$ such that $l_n=f^{t_n}gf^{-t_n}$ is loxodromic with 
\begin{equation}\label{Eq:ln}
\{ l_n^+,  l_n^-\}\in U_n\setminus \{ f^+, k_n^+, k_n^-, x_n k_n^+\}.
\end{equation}

\begin{figure}
  \centering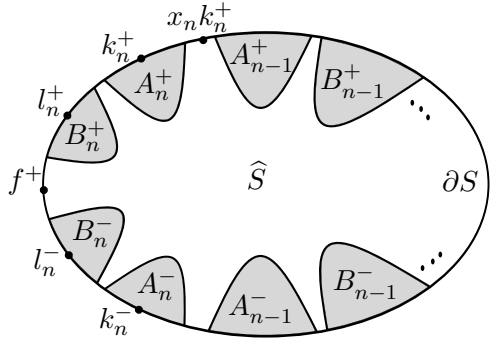\\
  \caption{Choosing the neighborhoods $A_n^+, A_n^-, B_n^+, B_n^-$ in $U_n$.}\label{fig}
\end{figure}

Further, (\ref{Eq:kn}) and (\ref{Eq:ln}) guarantee the existence of disjoint open neighborhoods $A_n^+, A_n^-, B_n^+, B_n^-$ of points $k_n^+$, $k_n^-$, $l_n^+$, $l_n^-$, respectively, that belong to $U_n$ and satisfy conditions (b), (c) for $i=n$ (see Fig. \ref{fig}). Clearly, these sets are disjoint from the sets $A^\pm_i$, $B^\pm _i$ constructed at all previous steps. Applying Lemma \ref{north-south} again, we can find $p_n, q_n \in \NN$  such that $a_n=k_n^{p_n}$ and $b_n=l_n^{q_n}$ satisfy condition (a). 

For brevity, we use the notation $y_i=b_ix_ia_i$ for all $i\in \NN$. Combining (a), (b), and the disjointness of the constructed sets, we obtain
\begin{equation}\label{Eq:yi}
 y_i(\widehat S\setminus A_i^-)=b_ix_ia_i (\widehat S\setminus A_i^-)\subseteq b_ix_i(A_i^+) \subseteq b_i(\widehat S\setminus B_i^-) \subseteq B_i^+ \subseteq \widehat S\setminus A_i^-,
\end{equation} 
for all $i \in \NN$. Similarly, we have
\begin{equation}\label{Eq:yi-1}
 y_i^{-1}(\widehat S\setminus B_i^+)= a_i^{-1}x_i^{-1}b_i^{-1}(\widehat S\setminus B_i^+)\subseteq a_i^{-1}x_i^{-1}(B_i^-) \subseteq a_i^{-1}(\widehat S\setminus A_i^+) \subseteq A_i^- \subseteq \widehat S\setminus B_i^+.
\end{equation}
Iterating (\ref{Eq:yi}) and (\ref{Eq:yi-1}) yields 
$y_i^{n}(\widehat S\setminus A_i^-)\subseteq B_i^+$ if $n>0$ and $y_i^{n}(\widehat S\setminus B_i^+)\subseteq A_i^-$ if $n<0$. Letting $C_i=A_i^+ \cup B_i^-$, we obtain
\begin{equation} \label{ping-pong}
y_i^{n}(\widehat S\setminus C_i) \subset C_i
\end{equation}
for all $i\in \NN$ and all $n\in \ZZ\setminus\{0\}$.

We now use the standard ping-pong argument (see \cite[Proposition 1.1]{Tits}) to show that the subgroup $F = \langle y_1, y_2, \ldots \rangle$ is free and does not contain $H$. Consider an arbitrary nonempty reduced word $w=y_{i_1}^{n_1}y_{i_2}^{n_2} \ldots y_{i_m}^{n_m}$ in the generators $\{y_1, y_2, \ldots\}$, where $n_j \in \ZZ \setminus\{0\}$ for $j \in \{1 \dots m\}$. Note that the sets $C_{i_j}$ are pairwise disjoint by construction and do not contain $f^+$ by (c). Using (\ref{ping-pong}), we obtain
$$
\begin{array}{rcl}
wf^+& = & y_{i_1}^{n_1}y_{i_2}^{n_2} \ldots y_{i_m}^{n_m}f^+ \in y_{i_1}^{n_1}y_{i_2}^{n_2} \ldots y_{i_m}^{n_m} (\widehat{S} \setminus C_{i_m}) \subset y_{i_1}^{n_1}y_{i_2}^{n_2} \ldots y_{i_{m-1}}^{n_{m-1}}(C_{i_m})\subset \\ 
&& y_{i_1}^{n_1}y_{i_2}^{n_2} \ldots y_{i_{m-1}}^{n_{m-1}}(\widehat{S} \setminus C_{i_{m-1}}) \subset \ldots \subset y_{i_1}^{n_1}(\widehat{S} \setminus C_{i_1}) \subset C_{i_1}.
\end{array}
$$
By (c), we have $wf^+ \neq f^+$; in particular, $w \neq 1$. This proves that $F$ is freely generated by $\{y_1, y_2, \ldots \}$. Moreover, since $f^+$ is fixed by $f$, we obtain that $f \notin F$; therefore, $F$ does not contain $H$.
\end{proof}

The following result is well-known (see, for example, \cite[Theorem 4.5]{Osi}, where a much more general fact is proved).

\begin{lem}\label{dichotomy}
Let $S$ be a hyperbolic space. Suppose that $G\le \operatorname{Isom} (S)$ is a subgroup of general type. Then every normal subgroup of $G$ either has bounded orbits in $S$ or is also of general type.
\end{lem}

We are now ready to prove our first main result.

\begin{proof}[Proof of Theorem \ref{main0}]
For the sake of contradiction, assume that $\Phi(G)$ has unbounded orbits and let $G=\langle x_1, x_2, \ldots\rangle$. By Lemma \ref{dichotomy}, $\Phi(G)$ is a subgroup of general type of $\operatorname{Isom} (S)$. Therefore, there exist two independent loxodromic elements $f,g\in \Phi(G)$. 

Applying Proposition \ref{main_prop} to the general type subgroup $H=\langle f,g\rangle $, we can find elements $a_1, b_1, a_2, b_2... \in H$ such that the subgroup $F=\langle b_1x_1a_1, b_2x_2a_2,... \rangle$ does not contain $H$. Obviously, $G = \langle f, g, b_1x_1a_1, b_2x_2a_2,...\rangle$. Since $\Phi(G)$ consists of non-generating elements, we obtain $G = \langle b_1x_1a_1, b_2x_2a_2,...\rangle = F$, which contradicts our assumption.
\end{proof}

Our next goal is to show how to apply Theorem \ref{main0} in certain particular cases. Recall that an \textit{$\mathbb R$-tree} is a geodesic metric space without non-trivial simple loops; an action of a group $G$ on an $\RR$--tree $T$ is said to be \textit{minimal} if $G$ does not fix setwise any proper subtree of $T$. For the definition and basic properties of acylindrically hyperbolic groups, we refer to \cite{Osi16}.

The first part of the corollary below is a minor generalization of the last claim of \cite[Theorem~1.17]{GG}; the second part was originally proved by Hull in \cite{H}.   
\begin{cor}\label{corPC}
\begin{enumerate}
 \item[(a)] Let $T$ be an $\mathbb R$ tree, and let $G\le \operatorname{Isom}(S)$. Suppose that $G$ is countable, of general type, and acts minimally on $T$. Then $\Phi(G)$ fixes $T$ pointwise; in particular, we have $\Phi(G)=\{ 1\}$ whenever the action of $G$ on $T$ is faithful. 

 \item[(b)] (Hull, {\cite[Theorem 1.9]{H}})\, For any countable acylindrically hyperbolic group $G$, $\Phi(G)$ is finite.
\end{enumerate}    
\end{cor}

\begin{proof}
(a)\; Let $T_0\subseteq T$ be the set of all points of $T$ fixed by $\Phi(G)$. By Theorem \ref{main0}, $\Phi(G)$ must have bounded orbits. It is well-known and easy to prove that every group acting on an $\mathbb R$-tree with bounded orbits must fix a point. Therefore, $T_0$ is non-empty. Since $T$ is an $\mathbb R$-tree, every two points in $T$ are connected by a unique geodesic. Therefore, $T_0$ is connected. Since $\Phi(G)$ is normal in $G$, $T_0$ is also $G$-invariant. Thus, $T_0$ is a $G$-invariant subtree of $T$. By minimality, we have $T_0=T$.

(b) By \cite[Theorem 1.2 and Lemma 7.1]{Osi}, every acylindrically hyperbolic group $G$ admits a general type action on a hyperbolic space $S$ such that the action of every infinite normal subgroup of $G$ is non-elementary (in particular, has unbounded orbits). Combining this with Theorem \ref{main0}, we conclude that $\Phi(G)$ must be finite. 
\end{proof}

Finally, we provide an example showing that Theorem \ref{main0} may fail for non-elementary actions.

\begin{ex}\label{ex}
Let $E$ be the minimal subfield of $\RR$ containing $\{ 2^q\mid q\in \QQ\}$. Consider
$$
H=\left\{ \left.\left(
\begin{array}{cc}
2^q & r\\
0 & 2^{-q} 
\end{array}
\right)\;\right|\; q\in \QQ,\; r\in E
\right\}.
$$
Clearly, $H$ is a group. The subgroup
$$
N=\left\{ \left.\left( 
\begin{array}{cc}
1 & r\\
0 & 1
\end{array}
\right)\;\right|\; r\in E
\right\}
$$
is a divisible normal subgroup of $H$ and $H/N\cong \QQ$. By \cite[Theorem 1]{HN}, every extension of a divisible abelian group by $\QQ$ has no maximal subgroups. Therefore, $\Phi(H)=H$. On the other hand, it is easy to verify that the standard action of $H$ on $\mathbb H^2$ by M\" obius transformations is non-elementary in the sense of Definition \ref{def}. (For example, this follows from the fact that $H$ contains both loxodromic and parabolic elements, which is impossible for elementary actions; we leave details to the reader.)
\end{ex}


\section{Lacunary hyperbolic groups}


Recall that a subgroup $H$ of a hyperbolic group $G$ is said to be \textit{elementary} if $H$ is virtually cyclic. It is well-known that every torsion-free virtually cyclic group is cyclic. This fact easily follows from Schur's theorem (see \cite{Sch} or \cite[Theorem 5.7]{Isa}). Alternatively, one can refer to a considerably more difficult theorem of Stallings \cite{Sta}, which asserts that finitely generated, torsion-free, virtually free groups are free. Thus, for subgroups of torsion-free hyperbolic groups, being non-elementary is the same as being non-cyclic, and we will use these words interchangeably. 

We will utilize the following version of \cite[Theorem 2]{Ols} in the particular case of torsion-free groups.

\begin{thm}[Olshanskii, {\cite{Ols}}]\label{Ols} Let $G$ be a torsion-free hyperbolic group, $H_1, \ldots, H_m$ a collection of non-cyclic subgroups of $G$. Further, for every $i=1, \ldots , m$, let $t_{i1}, \ldots, t_{in_i}$ be an arbitrary finite sequence of elements of $G$. For any finite subset $M$ of $G$, there exist elements $h_{ij}\in H_i$ ($i=1, \ldots, m$,  $j=1, \ldots, n_i$),  such that the quotient group 
\begin{equation}\label{Gbar}
\overline G=G\Big/\bll \big\{ t_{ij}h_{ij}\mid i=1, \ldots, m, \; j=1, \ldots, n_i\big\}\brr
\end{equation}
satisfies the following.
\begin{enumerate}
  \item[(a)] $\overline G$ is non-cyclic, hyperbolic, and torsion-free.
  \item[(b)] The restriction of the natural homomorphism $G \to \overline G$ to $M$ is injective.
\end{enumerate}
\end{thm}

Two remarks are in order. First, in \cite[Theorem 2]{Ols}, the subgroups $H_1, \ldots, H_m$ are required to satisfy a certain additional condition (specifically, to be Gromov subgroups in the terminology of \cite{Ols}), which is restated in terms of finite subgroups of $G$ in \cite[Theorem 1]{Ols}. For torsion-free groups, the latter condition is vacuously true. Thus, every non-elementary subgroup of a torsion-free hyperbolic group is a Gromov subgroup. (This fact is explicitly stated in the second sentence of the proof of \cite[Corollary 4]{Ols}.)

Second, the statement of \cite[Theorem 2]{Ols} does not provide an explicit presentation of the group $\overline G$.  However, the presentation of $\overline G$ can be obtained from the discussion of the proof on page 403; this presentation is exactly like ours, where elements $t_{ij}$ are denoted by $a_1, \ldots, a_n$. In \cite{Ols}, it is moreover assumed that $a_1, \ldots, a_n$ generate $G$, but this assumption is unnecessary for deriving the claims of Theorem \ref{Ols}.

More precisely, Olshanskii's version of the theorem asserts the existence of a quotient group $\overline G$ satisfying certain conditions (including our conditions (a) and (b)) such that the restriction of the natural homomorphism $G\to \overline G$ to each $H_i$ is surjective. The latter condition is ensured by taking  $\{t_{i1}, \ldots, t_{in_i}\}$ to be a generating set of $G$ (the same set for all $i$) and passing to a quotient group of the form (\ref{Gbar}), where $h_{i1}, \ldots, h_{in_i}$ are certain carefully chosen elements of $H_i$ (see the second paragraph on page 403 in \cite{Ols}). Repeating the same argument for arbitrary elements $t_{i1}, \ldots, t_{in_i}\in G$ verbatim yields Theorem \ref{Ols}.

For a group $G$ generated by a set $S$, we denote by $|\cdot |_S$ and $\d_S$ the corresponding word length and metric, respectively. Further, for every $r\ge 0$, we let 
$$
Ball_{G,S}(r) =\{ g\in G\mid |g|_S\le r\}.
$$
Finally, by $Cay(G,S)$ we denote the Cayley graph of $G$ with respect to $S$.

\begin{defn}
Let $G=\langle S\rangle$ and $G'$ be groups, and let $\alpha\colon G \rightarrow G^{'}$ be a homomorphism. The \emph{injectivity radius of $\alpha$} with respect to $S$ denoted $IR_S(\alpha)$, is the maximal $r$ such that the restriction of $\alpha$ to $Ball_{G,S}(r)$ is injective; if $\alpha $ is a monomorphism, we let $IR_S(\alpha)=\infty$.
\end{defn} 

We will make use of the following characterization of lacunary hyperbolic groups obtained in \cite{OOS}. Readers unfamiliar with asymptotic cones can take it as a definition.

\begin{thm}[\cite{OOS}, Theorem 1.1(3)]\label{lac_hyp}
A finitely generated group $G$ is lacunary hyperbolic if and only if $G$ is the direct limit of a sequence of groups $G_i = \langle S_i\rangle$ and epimorphisms
$$
G_1\xrightarrow{ \; \alpha_1 \; } G_2\xrightarrow{ \; \alpha_2 \; } \dots 
$$
such that the following conditions hold: 
\begin{enumerate}
\item[(a)] $\alpha_i(S_i) = S_{i+1}$;
\item[(b)] each $S_i$ is finite and the Cayley graph of $G_i$ with respect to $S_i$ is $\delta_i$-hyperbolic;
\item[(c)] $\lim\limits_{i\to \infty}\delta_i/IR_{S_i}(\alpha_i)=0$.
\end{enumerate}
\end{thm}

To prove Theorem \ref{main1}, we need a couple of auxiliary results. Recall that a subgroup $H$ of a group $G$ is called a \textit{retract} of $G$ if there is a homomorphism $\rho\colon G\to H$ such that the restriction of $\rho $ to $H$ is the identity map; the homomorphism $\rho$ is called a \textit{retraction}.

\begin{lem}\label{lem:retr}
A retract of a lacunary hyperbolic group is lacunary hyperbolic.
\end{lem}

\begin{proof}
We first show that a retract of a finitely generated group $G$ is isometrically embedded into $G$ for the appropriate choice of generating sets. This is well-known, but the argument is short and we provide it for the convenience of the reader. 

Let $G$ be a group generated by a finite set $S$, $H$ a retract of $G$, and $\rho\colon G\to H$ a retraction.  We define $Y=\rho (S)$ and $X=S\cup Y$. Clearly, we have $G=\langle X\rangle$, $H=\langle Y \rangle $, $\rho (X)=\rho(S)\cup \rho(Y)=Y\cup Y=Y$. Since $Y\subseteq X$, $|h|_X\le |h|_Y$ for all $h\in H$. Conversely, suppose that an element $h\in H$ can be written as a word $w$ of a certain length $n$ in generators $X\cup X^{-1}$. Applying $\rho$ to both sides of the equality $h=w$, we obtain that $h=\rho(h)$ can be written as a word of length $n$ in $\rho(X\cup X^{-1})=Y\cup Y^{-1}$. Thus, $|h|_Y\le |h|_X$. Combining this with the previously obtained inequality, we obtain that the group embedding $H\le G$ induces an isometric embedding of metric spaces $(H, \d_Y)\to (G, \d_X)$. 

To complete the proof of the lemma, it remains to refer to Theorem 3.18 in \cite{OOS}, which asserts that every quasi-isometrically embedded (in particular, every isometrically embedded) subgroup of a lacunary hyperbolic group is lacunary hyperbolic.
\end{proof}

For a group $G$ and a subgroup $N$, we denote by $C_G(N)$ the centralizer of $N$ in $G$. 
\begin{lem}\label{lem:no cyclic normal subgroups}
Suppose that $G$ is a torsion-free hyperbolic group, $H$ a non-cyclic subgroup of $G$, and $N$ a cyclic normal subgroup of $H$. Then $N=\{1\}$.
\end{lem}
\begin{proof} 
Assume that $N$ is non-trivial. Since $G$ is torsion-free, $N$ must be infinite. Recall that any infinite cyclic subgroup of a hyperbolic group has finite index in its centralizer (see, for example, \cite[Theorem 34, Chapter 8]{GdlH}). Thus, $[C_G(N):N]<\infty$. Further, since $N\triangleleft H$, we have a homomorphism $H\to Aut(N)$ with kernel $C_H(N)$. Since $N$ is cyclic, $Aut(N)$ is of order $2$, and hence $[H:C_H(N)] \leq 2$. Summarizing, we obtain that $H$ is virtually cyclic. However, this contradicts the well-known property that every non-elementary subgroup of a hyperbolic group contains a free subgroup of rank $2$ \cite[Theorem 37, Chapter 8]{GdlH}.
\end{proof}

We are now ready to prove Theorem \ref{main1}.

\begin{proof}[Proof of Theorem \ref{main1}]
Let $Q$ be not a finite cyclic group with $\Phi(G)=\{1\}$ given by a presentation 
$$
Q=\langle X \mid R_1, R_2, \ldots \rangle,
$$
where $X=\{ x_1, \ldots, x_n\}$. If $Q$ is infinite and lacunary hyperbolic, we can simply take $L=Q$. In what follows, we assume that $Q$ is either finite or not lacunary hyperbolic. 

We denote by $F$ the free group with basis $X\cup \{a,b\}$ and let $X_0=\emptyset, X_1, X_2 \ldots $ be the list of all finite subsets of $F$. The proof of the theorem will consist of two steps. At the first step, we will inductively construct a sequence of groups and epimorphisms
\begin{equation}\label{eq:Gi}
F\xrightarrow{ \; \alpha_0 \; } G_0\xrightarrow{ \; \alpha_1 \; }  G_1 \xrightarrow{ \; \alpha_2 \; }  \ldots
\end{equation}
that satisfy conditions \ref{item: torsion free, non cyclic, hyperbolic}--\ref{item: a,b are non-generators} presented below for all $i\in \mathbb N\cup \{ 0\}$. For brevity, we denote by $\gamma_i\colon F\to G_i$ the composition $\alpha_0 \circ \cdots \circ \alpha _i$ and let $a_i=\gamma_i(a)$, $b_i=\gamma_i(b)$, $N_i=\langle a_i,b_i\rangle$, and $S_i =\gamma_i(X\cup \{ a, b\})$. 

\begin{enumerate}[(a)]
\item The group $G_i$ is torsion-free, non-cyclic, and $Cay(G_i, S_i)$ is $\delta_i$-hyperbolic for some $\delta_i\ge 0$. \label{item: torsion free, non cyclic, hyperbolic}
\item $a_i\ne 1$ and $b_i\ne 1$ in $G_i$. Moreover, if $i>0$, then $IR_{S_{i-1}}(\alpha_{i})> i\cdot \delta_{i-1}$. \label{item: inj radius is large}
\item The subgroup $N_i$ is normal in $G_i$ and the map sending $\gamma_i(x)N_i$ to $x$  for all $x\in X$ extends to an isomorphism of $G_i/N_i$ and the group $$Q_i=\langle X \mid R_1,R_2,\dots, R_{i}\rangle$$ (for convenience, we assume $Q_0$ to be the free group with basis $X$). \label{item: quotient canonically isom to Qi}

\item Suppose that $i>0$ and, for some $k\le i$, the natural image of $X_k$ in $G_{i-1}/N_{i-1}$ generates $G_{i-1}/N_{i-1}$ and $\langle \gamma_{i-1}(X_k)\rangle \cap N_{i-1} \neq \{1\}$; then $N_{i}\le  \langle \gamma_{i}(X_k)\rangle$.\label{item: a,b are non-generators}
\end{enumerate}

At the second step, we will show that the limit of the sequence (\ref{eq:Gi}) defined by 
\begin{equation}\label{L}
L= F/\bigcup _{i\in \mathbb{N}} \ker (\gamma_i)
\end{equation}
satisfies the requirements of the theorem. The reader may benefit from keeping in mind the following outline: conditions \ref{item: torsion free, non cyclic, hyperbolic} and \ref{item: inj radius is large} will ensure that $L$ is an infinite lacunary hyperbolic group; condition \ref{item: quotient canonically isom to Qi} implies that $L/ N \cong Q$, where $N$ is the image of $N_0$ in $L$; finally, condition \ref{item: a,b are non-generators} is used to show that $N=\Phi(L)$.  The attentive reader may notice that the inductive step will not make use of  \ref{item: a,b are non-generators};
this condition will only be used at the second step of the proof.

\medskip

\noindent\emph{Step 1: The inductive construction.} We now construct the sequence (\ref{eq:Gi}). Let $G_0$ be the non-cyclic, hyperbolic, torsion-free group obtained by applying Theorem \ref{Ols} to the group $F$, single subgroup $H=\langle a,b \rangle$, collection of elements 
$\{xax^{-1}, \; xbx^{-1}\mid x\in X\cup X^{-1}\big\}$, and  subset $M=Ball_{F,X\cup \{ a, b\}}(1)$. 
Thus, $G_0$ is given by the presentation 
$$
G_0= \big\langle X, a, b \mid xax^{-1}=u_{x},\; xbx^{-1}=v_x\;\; \forall \, x\in X\cup X^{-1}\big \rangle
$$
where $u_x$ and $v_x$ are some non-trivial words in the alphabet $\{ a^{\pm 1}, b^{\pm 1}\}$ depending on $x$. We also define $\alpha_0\colon F\to G_0$ to be the natural quotient map. The relations of $G_0$ guarantee that  $N_0\triangleleft G_0$ is normal and the map sending $\gamma_0(x)N_0$ to $x$  for all $x\in X$ extends to an isomorphism $G_0/N_0\cong Q_0$. Thus, conditions  \ref{item: torsion free, non cyclic, hyperbolic} and \ref{item: quotient canonically isom to Qi} hold. Our choice of $M$ guarantees that $a_0\ne 1$ and $b_0\ne 1$ in $G$. Note that the second part of the condition \ref{item: inj radius is large} and condition \ref{item: a,b are non-generators} is vacuously true for $i=0$.

Suppose now that groups $G_0,\dots,G_i$ and epimorphisms $\alpha_0,\dots,\alpha_i$ satisfying  \ref{item: torsion free, non cyclic, hyperbolic}--\ref{item: a,b are non-generators} have already been constructed for some $i\in \mathbb{N}\cup \{0\}$. Let $K$ be the set of all integers $0\leq k \leq i+1$ satisfying the following conditions.
\begin{enumerate}
    \item[($\ast$)] The natural image of $X_k$ in $G_i/N_i$ generates $G_i/N_i$.
    \item[($\ast\ast$)]  $\langle \gamma_i(X_k)\rangle  \cap N_i \neq \{1\}$.
\end{enumerate}
We note that for every $k\in K$, the subgroup $\langle \gamma_i(X_k)\rangle \cap N_i\leq G_i$ is non-cyclic. Indeed, by ($\ast$) and condition \ref{item: quotient canonically isom to Qi} of the inductive hypothesis, the subgroup $\langle \gamma_i(X_k)\rangle$ projects onto $Q_i$. Therefore, the subgroup $\langle \gamma_i(X_k)\rangle$ of the hyperbolic group $G_i$ is non-cyclic and its normal subgroup $\langle \gamma_i(X_k)\rangle\cap N_i$ must be non-cyclic by Lemma \ref{lem:no cyclic normal subgroups} and  ($\ast\ast$).

Let $m=|K|$ and let $H_1,\dots,H_m$ be an enumeration of subgroups in the set $\{ \langle \gamma_i(X_k)\rangle\cap N_i  \colon k\in K\}$. Note that all the subgroups $H_1,\dots,H_m$ are non-cyclic as explained above. Let also $H_{m+1}=N_i$. The first part of condition \ref{item: inj radius is large} guarantees that $H_{m+1}$ is non-trivial and, therefore, also non-cyclic by Lemma \ref{lem:no cyclic normal subgroups}. We apply Theorem \ref{Ols} to the torsion-free hyperbolic group $G_i$, the non-cyclic subgroups $H_1,\dots, H_{m+1}$, the elements 
$$
t_{11}=\cdots=t_{m1}=a_i,\;\;\; t_{12}=\cdots=t_{m2}=b_,\;\;\; t_{m+1,1}=\gamma_i(R_{i+1})
$$
(in the notation of Theorem \ref{Ols}, we have $n_1=\cdots=n_m=2$  and $n_{m+1}=1$ here), and the finite set $$M=Ball_{G_{i},S_{i}}((i+1)\cdot \delta_{i} +1).$$

Let 
\begin{equation}\label{Gquot}
G_{i+1}= G_i \Big/\bll a_ih_{11},\; b_ih_{12},\; \ldots ,\; a_ih_{m1},\; b_ih_{m2}, \;\gamma_i(R_{i+1})h_{m+1,1}\brr 
\end{equation}
be the resulting quotient group, where $h_{\ell 1}, h_{\ell 2}\in H_\ell$ ($\ell =1,\dots,m$), and $h_{m+1,1}\in H_{m+1}$. Since $G_{i+1}$ is non-cyclic, hyperbolic, and torsion-free, condition \ref{item: torsion free, non cyclic, hyperbolic} holds for $i+1$.  Further, let $\alpha_{i+1}\colon G_i\to G_{i+1}$ be the natural homomorphism. Since $\alpha_{i+1}$ is injective on $M$, we obtain \ref{item: inj radius is large} for $i+1$.

The subgroup $N_{i+1}$ is normal in $G_{i+1}$ being the image of the normal subgroup $N_i$ of $G_i$. Moreover, since $H_1,\dots,H_{m+1}\le N_i$, we have an isomorphism
$$
G_{i+1}/N_{i+1} \cong (G_i/N_i)\Big/\bll \gamma_i(R_{i+1})N_i\brr
$$
that identifies elements $\gamma_{i+1}(x)N_{i+1}$ in the left hand side with the cosets $(\gamma_i(x)N_i)\ll \gamma_i(R_{i+1})N_i\rr$ in the right hand side for all $x\in X$. Combining this with the isomorphism $G_i/N_i\cong Q_i$ given by condition \ref{item: quotient canonically isom to Qi} of the inductive hypothesis and the observation that the presentation of $Q_{i+1}$ is obtained from that of $Q_i$ by adding the relation $R_{i+1}=1$, we obtain an isomorphism $G_{i+1}/N_{i+1}\cong Q_{i+1}$ satisfying \ref{item: quotient canonically isom to Qi} for $i+1$.

Finally, let us verify condition \ref{item: a,b are non-generators}. Consider any $k\leq i+1$. If $k\notin K$, \ref{item: a,b are non-generators} vacuously holds for this $k$. If $k\in K$, there exists $\ell$ such that $H_\ell =\langle \gamma_i(X_k)\rangle \cap N_i$. By (\ref{Gquot}), we have $\alpha_{i+1}(a_i), \alpha_{i+1}(b_i)\in\alpha_{i+1}(H_\ell)$. Therefore, $N_{i+1}\leq \langle \gamma_{i+1}(X_k)\rangle$. Thus, \ref{item: a,b are non-generators} holds for $i+1$. This completes the inductive construction.

\medskip

\noindent\emph{Step 2: The limit group.} We now define the limit group $L$ by (\ref{L}). Condition \ref{item: torsion free, non cyclic, hyperbolic} of the inductive construction implies that $L$ is non-trivial and torsion-free (in particular, infinite). By conditions \ref{item: torsion free, non cyclic, hyperbolic}, \ref{item: inj radius is large}, and Theorem \ref{lac_hyp}, $L$ is lacunary hyperbolic.  

Let $\gamma\colon F\to L$ be the natural projection and let $N=\langle \gamma(a),\gamma(b)\rangle$. Condition \ref{item: quotient canonically isom to Qi} of the inductive construction implies that the map $\gamma(x)N\mapsto x$ for all $x\in X$ extends to an isomorphism $L/N\cong Q$. 

To conclude the proof of the theorem, it remains to show that $N=\Phi(L)$. The assumption that $\Phi(Q)=\{1\}$ implies that $\Phi(L) \leq N$ since the image of $\Phi(L)$ in $L/N$ is contained in $\Phi(L/N)\cong \Phi(Q)$. To prove the opposite inclusion $N\le \Phi(L)$, it suffices to show that the generators $\gamma (a)$ and $\gamma(b)$ of $N$ can be excluded from every finite generating set of $L$. 

To this end, let $C$ be a finite subset of $F$ such that $\gamma(C)$ generates $L$. By construction, $C\setminus \{a,b\}=X_k$ for some $k\in \mathbb{N}$. Since $\gamma(C)$ generates $L$, the image of $\gamma(X_k)$ in $L/N$ generates $L/N$. Since $F$ is finitely generated, it follows that the natural image of $\gamma_i(X_k)$ in $G_i/N_i$ generates $G_i/N_i$ for all sufficiently large $i\in \NN$.

Suppose that $\langle \gamma(X_k)\rangle \cap N= \{1\}$. Then $\langle \gamma(X_k)\rangle\cong L/N\cong Q$, and so $Q$ is a retract of $L$. Hence, $Q$ is lacunary hyperbolic by Lemma \ref{lem:retr}. Since $Q$ is torsion-free, it must be infinite, which contradicts our assumptions on $Q$. Thus, $\langle \gamma(X_k)\rangle \cap N\neq \{1\}$. It follows that $\langle \gamma_i(X_k)\rangle \cap N_i\neq \{1\}$ for all sufficiently large $i\in \NN$. 

Let $i\in \mathbb{N}$ be large enough so that $i\ge k$,  the image of $\gamma_i(X_k)$ in $G_i/N_i$ generates $G_i/N_i$, and $\langle \gamma_i(X_k)\rangle \cap N_i\neq \{1\}$. By condition \ref{item: a,b are non-generators}, $N_{i+1}\leq \langle \gamma_{i+1}(X_k)\rangle$. In particular, $\gamma_{i+1}(X_k)$ generates $G_{i+1}$ and therefore $\gamma(X_k)$ generates $L$. It follows that $\gamma(C\setminus \{a,b\})$ generates $L$, thus concluding the proof.
\end{proof}

\section{Open questions}

We conclude with a brief discussion of some open problems prompted by our results.

\begin{prob}
Let $S$ be a hyperbolic space and let $G \leq Isom(S)$ be a general type subgroup.
\begin{enumerate}
    \item[(a)] Are the orbits of $\Phi(G)$ bounded?
    \item[(b)] Does $G$ contain a maximal subgroup?
\end{enumerate} 
\end{prob}

Clearly, an affirmative answer to (a) implies an affirmative answer to (b). Even when $S$ is a tree, both questions remain open (for a detailed discussion of this particular case, we refer to \cite{All} and references therein). A version of the second question for amalgamated free products can also be found in \cite[Problem 17.114]{Kou}.  Theorem \ref{main0} answers the first question in the affirmative if $G$ is countable. However, the use of countability is indispensable in our proof.

A straightforward attempt to provide negative answers to both questions involves considering the full automorphism group $Aut(T_n)$ of an $n$-regular tree $T_n$ for some $n\ge 3$. In this case, the completion $\widehat T_n$ is compact and there is ``no room" in $\widehat T$ to extend our proof of Proposition \ref{main_prop} to uncountable subsets of $Aut(T_n)$ through transfinite induction. However, the conclusion of Theorem \ref{main0} remains true in this case. Indeed, the action of $Aut(T_n)$ on $\partial T_n$ is easily seen to be $3$-transitive. This implies that the stabilizer $Stab_G(x)$ of every point $x\in \partial T_n$ acts $2$-transitively and, in particular, primitively on $\partial T_n$; the latter condition is equivalent to the maximality of $Stab_G(x)$ in $G$. Thus, we obtain 
$$
\Phi(Aut(T_n))\le \bigcap\limits_{x\in \partial T_n} Stab_G(x)=\{ 1\}.
$$

Further, we do not know the answer to the following, even in the particular case of countable groups.

\begin{prob}
    Does there exist an invariably generated group $G$ that admits a general type action on a hyperbolic space? Can such a group be finitely invariably generated?
\end{prob}

\vspace{5mm}

\noindent \textbf{Gil Goffer: }Department of Mathematics, University of California, San Diego 92093, U.S.A.
 
\noindent E-mail: \emph{ggoffer@ucsd.edu}

\bigskip

\noindent \textbf{Denis Osin: }Department of Mathematics, Vanderbilt University, Nashville 37240, U.S.A.

\noindent E-mail: \emph{denis.v.osin@vanderbilt.edu}

\bigskip

\noindent \textbf{Ekaterina Rybak: }Department of Mathematics, Vanderbilt University, Nashville 37240, U.S.A.

\noindent E-mail: \emph{ekaterina.rybak@vanderbilt.edu}

\end{document}